\documentclass{amsart}

\usepackage{amssymb,amsfonts,amsthm}
\usepackage{enumerate}
\usepackage{tikz-cd}
\usepackage{colonequals}
\newcommand{\Z}{\mathbb{Z}}
\newcommand{\F}{\mathbb{F}}
\newcommand{\G}{\mathbb{G}}

\newcommand{\Hom}{\underline{\mathrm{Hom}}}
\newcommand{\Isom}{\underline{\mathrm{Isom}}}
\newcommand{\overbar}[1]{\mkern 1.3mu\overline{\mkern-1.3mu#1\mkern-1.3mu}\mkern 1.3mu}
\DeclareMathOperator{\spec}{Spec}

\newtheorem{theorem}{Theorem}[section]
\newtheorem{lemma}[theorem]{Lemma}
\newtheorem{proposition}[theorem]{Proposition}
\theoremstyle{definition}
\newtheorem{definition}[theorem]{Definition}
\newtheorem{example}[theorem]{Example}

\theoremstyle{remark}
\newtheorem{remark}[theorem]{Remark}
\numberwithin{equation}{section}

\begin{document}

\title{Full Level Structure on Some Group Schemes}

\author{Chuangtian Guan}
\address{Department of Mathematics, Michigan State University, 619 Red Cedar Road, East Lansing, MI 48824, USA}
\email{guanchua@msu.edu}


\date{\today}

\begin{abstract}
We give a definition of full level structure on group schemes of the form $G\times G$, where $G$ is a finite flat commutative group scheme of rank $p$ over a $\mathbb{Z}_p$-scheme $S$ or, more generally, a truncated $p$-divisible group of height $1$. We show that there is no natural notion of full level structure over the stack of all finite flat commutative group schemes.
\end{abstract}

\maketitle

\section{Introduction}
Throughout this paper, all group schemes are assumed to be finite flat and commutative over the base. Let $H$ be such a group scheme of rank $p^l$ over a $\Z_p$-scheme $S$ (``rank $p^l$" means that $\mathcal{O}_H$ is a locally free $\mathcal{O}_S$-algebra of rank $p^l$).  Suppose that $H$ is annihilated by $p^r$ and $H[\frac{1}{p}]\colonequals H\times_SS[\frac{1}{p}]$ is \'etale-locally isomorphic to the constant group scheme $(\underline{\Z/p^r\Z})^g$ for some $r$ and $g$. This happens, for example, when $H$ is the $p^r$-torsion of some abelian variety of dimension $g/2$.

Let $\mathrm{Hom}_S((\mathbb{Z}/p^r\mathbb{Z})^g,H)$ be the functor from the category of $S$-schemes $\mathbf{Sch}_S$ to the category of abelian groups $\mathbf{Ab}$, defined by $$\mathrm{Hom}_S((\mathbb{Z}/p^r\mathbb{Z})^g,H)(T)\colonequals\mathrm{Hom}_{gp}((\mathbb{Z}/p^r\mathbb{Z})^g,H(T)).$$ We will use $\Hom_S((\mathbb{Z}/p^r\mathbb{Z})^g,H)$ to denote the representing scheme. Since $H$ is annihilated by $p^r$, the representing scheme is just $H^g$. The general linear group $\mathrm{GL}_g(\Z/{p^r}\Z)$ has a natural right action on $\Hom_S((\mathbb{Z}/p^r\mathbb{Z})^g,H)$ by acting on $(\mathbb{Z}/p^r\mathbb{Z})^g$ by precomposition.

The problem we consider is to give a notion of full level structure on $H$. We expect it to be a closed subscheme of $\Hom_S((\mathbb{Z}/p^r\mathbb{Z})^g,H)$, which we denote by $\Hom^\ast_S((\mathbb{Z}/p^r\mathbb{Z})^g, H)$, satisfying:

\begin{enumerate}
\item $\Hom^\ast_S((\mathbb{Z}/p^r\mathbb{Z})^g, H)$ is {\sl flat} over $S$ and of rank $|\mathrm{GL}_g(\Z/{p^r}\Z)|$.

\item $\Hom^\ast_S((\mathbb{Z}/p^r\mathbb{Z})^g, H)$ is $\mathrm{GL}_g(\Z/{p^r}\Z)$-invariant under the right $\mathrm{GL}_g(\Z/{p^r}\Z)$-action on $\Hom_S((\mathbb{Z}/p^r\mathbb{Z})^g, H)$. Away from characteristic $p$ we have an identification $\Hom^\ast_{S[\frac{1}{p}]}((\mathbb{Z}/p^r\mathbb{Z})^g, H[\frac{1}{p}])=\Isom_{S[\frac{1}{p}]}((\mathbb{Z}/p^r\mathbb{Z})^g, H[\frac{1}{p}])$ as closed subschemes of $\Hom_{S[\frac{1}{p}]}((\mathbb{Z}/p^r\mathbb{Z})^g, H[\frac{1}{p}])$.

\item When identifying $\Hom_S((\mathbb{Z}/p^r\mathbb{Z})^g, H)\times_S T$ with $\Hom_T((\mathbb{Z}/p^r\mathbb{Z})^g,H_T)$ in the natural way, we have $\Hom^\ast_{S}((\mathbb{Z}/p^r\mathbb{Z})^g, H)\times_S T=\Hom^\ast_{T}((\mathbb{Z}/p^r\mathbb{Z})^g,H_T)$ as closed subschemes, for any $S$-scheme $T$.
\end{enumerate}

We also expect our definition to coincide with the intuitive definition for some familiar group schemes. For example, for $H=\mu_{p^r}$, we expect $\Hom^\ast_{\mathbb{Z}}(\mathbb{Z}/p^r\mathbb{Z}, H)$ to be the closed subscheme of $\mu_{p^r}$ defined by the cyclotomic polynomial $$\Phi_{p^r}(x)\colonequals\dfrac{x^{p^r}-1}{x^{p^{r-1}}-1}=x^{(p-1)p^{r-1}}+x^{(p-2)p^{r-1}}+\cdots+1.$$ When $H$ is the constant group scheme $(\underline{\Z/p^r\Z})^g$, the resulting full level structure $\Hom^\ast_{\mathbb{Z}}((\mathbb{Z}/p^r\mathbb{Z})^g, H)$ should be $\underline{\mathrm{GL}_g(\Z/p^r\Z)}\subset \underline{\mathrm{Mat}_g(\Z/p^r\Z)}$.

The motivation for giving a well-behaved notion of full level structure comes from the study of integral models of Shimura varieties. For example, for modular curves, finding an integral model of the modular curve $X(p^r)$ essentially amounts to finding a flat model of full level structure on the $p^r$-torsion of elliptic curves. This is done by Katz and Mazur in their book \cite{KM}: following an idea of Drinfeld in \cite{Dr}, Katz and Mazur consider the case when $H$ can be embedded into a curve. In this case a set of sections $\{P_1, \ldots,  P_r\}$ of $H$ is defined to be a ``full set of sections", if the points generate the group $H$ as Cartier divisors. Using this notion, the full level structure on $H$ is defined to be the maps in $\mathrm{Hom}_S((\mathbb{Z}/p^r\mathbb{Z})^g,H)$ whose image forms a full set of sections. As a scheme, $\Hom^\ast_S((\mathbb{Z}/p^r\mathbb{Z})^g, H)$ can be also described as the closed subscheme of $\Hom_S((\mathbb{Z}/p^r\mathbb{Z})^g,H)$ cut out by the Cartier divisor equation 
\begin{center}
$\displaystyle H=\sum\limits_{x\in(\Z/p^r\Z)^g}[h(x)]$
\end{center}
where $h$ is the universal homomorphism. Katz and Mazur's construction, for example, gives a definition of full level structure on $\underline{\Z/p\Z}\times \mu_p$, as it is the $p$-torsion of an ordinary elliptic curve. They also suggest a natural generalization of their construction, given by ``$\times$-homomorphisms" \cite[Appendix of Chapter 1]{KM}, that can be defined for general group schemes. Unfortunately, the notion of $\times$-homomorphisms is deficient because the resulting closed subscheme is generally not flat over the base. Such a negative result has been observed by Chai and Norman in \cite[Appendix 2]{CN}. For example, the nonflatness for $\times$-homomorphisms even happens on $\mu_p\times \mu_p$.

As an improvement, Wake gives in \cite{Wa} a good definition in the case of $H=\mu_p \times \mu_p$ over $\spec \Z$. By using a notion of ``primitive elements", he defines the full level structure, called ``scheme of full homomorphisms", to be cut out by the condition that all nontrivial linear combinations of rows and columns of the universal homomorphism are primitive. Alternatively, Wake also gives another level structure, called ``KM+D" level structure, short for Katz-Mazur + Dual. The notion of KM+D level structure is defined by requiring both universal homomorphism and its dual being $\times$-homomorphisms as defined by Katz and Mazur. Wake proves that in the case $\mu_p\times \mu_p$, the KM+D level structure coincides with his original notion of full homomorphisms. Unfortunately, in general the ``KM+D" level structure does not give a flat scheme over the base. For example, it is observed in \cite[Example 4.8]{Wa} that $\Hom^{\mathrm{KM+D}}_{\F_2}((\mathbb{Z}/2\mathbb{Z})^2,(\alpha_2)^2)$ has larger rank than expected.

In this paper, we give a definition of full level structure for $H$ of the form $H=G\times G$, where $G$ is a rank $p$ group scheme over a $\Z_p$-scheme $S$. When $G$ is $\mu_p$, our definition coincides with the one in \cite{Wa}. The idea of our construction is to generalize Wake's ``rows-and-columns" construction to a general group scheme $G$ using Kottwitz-Wake's notion of primitive elements \cite{KW}. In \cite{KW} the authors give a notion of primitive elements which is well-behaved, even for general $p$-divisible groups. Using this notion, our full level structure will be cut out by the condition that rows and columns of the universal homomorphism are linearly independent, as in Wake's construction. The precise description and properties are discussed in Section 3. The main point is that this construction gives a flat model. We show this by using Oort--Tate theory to reduce to Wake's result.

One might also expect the following naturality condition:
\begin{enumerate}
  \setcounter{enumi}{3}
  \item For any group scheme isomorphism $H\xrightarrow{\sim} H^\prime$, the induced isomorphism $\Hom_S((\mathbb{Z}/p^r\mathbb{Z})^g,H)\to \Hom_S((\mathbb{Z}/p^r\mathbb{Z})^g,H^\prime)$ restricts to an isomorphism $\Hom_S^\ast((\mathbb{Z}/p^r\mathbb{Z})^g,H)\to \Hom_S^\ast((\mathbb{Z}/p^r\mathbb{Z})^g,H^\prime)$.
\end{enumerate}
This condition $(4)$ can be interpreted as saying the notion of full level structure is defined over the stack. Unfortunately, it turns out that in general there is no level structure on $H$ satisfying all conditions $(1)-(4)$. After we wrote a first draft of this paper, Wake informed me that he and Kottwitz had proven a similar result in unpublished work. He pointed out that the construction cannot extend to the stack. We discuss this result in Section 6 and include their example there.

\textbf{Acknowledgment: }I would like to thank my advisor George Pappas for his support and encouragement. I thank Preston Wake for very helpful conversations and useful suggestions. I also thank the reviewers for patient reading and helpful suggestions.

\section{Review of the Oort--Tate theorem}

In \cite{OT}, Oort and Tate determine the structure of all finite flat commutative group schemes of rank $p$ over a $\Z_p$-scheme $S$. They prove that such a group scheme is given by a triple $(\mathcal{L},u,v)$ where $\mathcal{L}\in \text{Pic}(S)$, $u\in\Gamma(S,\mathcal{L}^{\otimes{(p-1)}})$ and $v\in\Gamma(S,\mathcal{L}^{\otimes{(1-p)}})$ satisfying $u\otimes v=w_p$, where $w_p$ is a constant in $p\Z_p^\times$. Specifically, when $S=\spec A$ where $A$ is a local ring, the line bundle $\mathcal{L}$ is trivial. Therefore to give an rank $p$ group scheme over such $S$, it suffices to give two elements $u,v\in A$ satisfying $uv=w_p$. For such a pair $(u,v)$, the corresponding Hopf algebra is $\spec A[x]/(x^p-ux)$ with the comultiplication 
\[
m^\ast(x)=1\otimes x+x\otimes 1 +\dfrac{1}{1-p}\sum\limits_{i=1}^{p-1}\dfrac{vx^i\otimes x^{p-i}}{w_iw_{p-i}}.
\]
 Here $w_i$'s are constants in $\Z_p$ with $w_1,\cdots, w_{p-1}\in \Z_p^{\times}$. The constants $w_1,\cdots, w_{p-1}$ satisfy that $w_i\equiv i! \mod p$ and $w_p=pw_{p-1}$. For more details on the $w_i$'s, see \cite[page 10]{OT}.

Haines and Rapoport express this result using stack language in \cite[Theorem 3.3.1]{HR}. For convenience, we give the result here:

\begin{theorem}[\cite{HR}]
The $\Z_p$-stack $OT$ of finite flat commutative group schemes of rank $p$, satisfies the following properties:
\begin{enumerate}[(i)]
\item $OT$ is an Artin stack isomorphic to $$[(\spec \Z_p[s,t]/(st-w_p))/ \mathbb{G}_m].$$
The action of $\G_m$ is given by $\lambda\cdot(s,t)=(\lambda^{p-1}s, \lambda^{1-p}t)$ and $w_p$ is a constant in $p\Z_p^\times$.
\item The universal group scheme $\mathcal{G}$ over $OT$ is $$\mathcal{G}=[(\spec_{OT} \mathcal{O}[x]/(x^p-tx))/ \mathbb{G}_m].$$
The action of $\G_m$ is given by $\lambda\cdot x=\lambda x$.
\end{enumerate}
\end{theorem}

\section{Full level structure}\label{sec3}

The first step in defining the full level structure on $G\times G$ is defining ``primitive elements". This is done by Kottwitz and Wake in \cite{KW}.

\begin{definition}[\cite{KW}]\label{primitive}
Let $H$ be a finite flat group scheme over a base scheme $S$. The scheme of primitive elements $H^{\times}$ is the closed subscheme of $H$ with the defining ideal sheaf given as the annihilator of the augmentation ideal sheaf.
\end{definition}

An important example is the Oort--Tate group scheme $G=\spec A[x]/(x^p-tx)$, where $A$ is a $\Z_p$-algebra. The augmentation ideal is $(x)$. Thus $G^\times$ is defined by the ideal $(x^{p-1}-t)$, coinciding with the scheme of generators defined in \cite{HR}.

Another example is $G\times G$. Its underlying algebra is $A[x,y]/(x^p-tx, y^p-ty)$ with the augmentation ideal $(x,y)$. The annihilator of $(x,y)$ is the intersection of the annihilator of $x$, which is $(x^{p-1}-t)$, and the annihilator of $y$, which is $(y^{p-1}-t)$. Suppose $f\in (x^{p-1}-t)\cap(y^{p-1}-t)$. Since $f\in (x^{p-1}-t)$, we may assume $f=(x^{p-1}-t)g$ with $g\in A[y]$ of degree less than $p$. Similarly $f=(y^{p-1}-t)h$ with $h\in A[x]$ and $\deg h<p$. Note that in this ring $A[x,y]/(x^p-tx, y^p-ty)$, every element can be uniquely written as a polynomial with the degrees of $x$ and $y$ being less than $p$. Because of the unique expression, we have $(x^{p-1}-t)g=(y^{p-1}-t)h$  as polynomials in $A[x,y]$. By comparing the coefficients, we can get $(y^{p-1}-t)|g$. Therefore we have $$(x^{p-1}-t)\cap(y^{p-1}-t) = \left((x^{p-1}-t)(y^{p-1}-t)\right)$$ and thus the scheme of primitive elements in $G^2$ is $$(G^2)^\times=\spec A[x,y]/\left((x^{p-1}-t)(y^{p-1}-t)\right).$$ See also \cite[Section 3.8]{KW}.

Now we consider the operation on the points of $\mathrm{Hom}_{S}((\mathbb{Z}/p\mathbb{Z})^2,G^2)=G^4$ (as functors). We will identify $G^4(T)$ with $\mathrm{Mat}_2(G(T))$, the additive group of $2\times 2$ matrices with entries in $G(T)$. On each entry $\mathrm{Hom}_S(\mathbb{Z}/p\mathbb{Z},G)(T)=G(T)$, there is a natural addition arising from the group structure of $G$. We denote this addition by $\dot{+}$, to distinguish it from the addition on $\mathcal{O}_G$. For simplicity, for any $f\in G(T)$, let $[m]f$ be $f\dot{+}f\dot{+}\cdots\dot{+}f$, the sum of $m$ copies of $f$. Since the Oort--Tate Group is annihilated by $p$, the operation $[m]$ only depends on $m$ modulo $p$.

\begin{example}	
	Let $S=\spec \Z_p$ and $G=\spec \Z_p[x]/(x^p-x)$ with comultiplication $m^\ast(x)=1\otimes x+x\otimes 1 +\frac{1}{1-p}\sum_{i=1}^{p-1}\frac{w_px^i\otimes x^{p-i}}{w_iw_{p-i}}$. This is obtained by taking $u=1$ and $v=w_p$ from Section 2. Let $T=\spec \Z_p$. In $G(T)$, let $\chi(j)\in \mathrm{Hom}_S(\mathbb{Z}/p\mathbb{Z},G)(T)=G(T)$ be the map sending $x$ to $\chi(j)$, where $\chi$ is the Teichm\"uller character and let $\chi(0)=0$. Since the elements in $G(T)$ are closed under the group action, we have $([j](1))^p-[j](1)=0$. On the other hand, by the definition of the comultiplication of $G$, we have $[j](1)\equiv j \mod p$. Therefore $[j](1)=\chi(j)$. From $[j](1)\dot{+}[k](1)=[j+k](1)$, we get a useful equation:
	\begin{equation}\label{eq5}
	\chi(j+k)=\chi(j)+\chi(k)+\dfrac{1}{1-p}\sum\limits_{i=1}^{p-1}\dfrac{w_p\chi(j^i) \chi(k^{p-i})}{w_iw_{p-i}}.
	\end{equation}
	
	In fact, $G$ is isomorphic to the constant group scheme $\underline{\Z/p\Z}_{S}=\spec \Z_p^{\Z/p\Z}$. The Hopf algebra isomorphism between $\Z_p[x]/(x^p-x)$ and $\Z_p^{\Z/p\Z}$ is given by $x\mapsto \sum \chi(i)e_i$ and $e_i\mapsto \lambda(i)\prod_{j\ne i}(x-\chi(j))$, where $\lambda(0)=-1$ and $\lambda(i)=\frac{1}{p-1}$ otherwise. To see this, we first easily observe that the maps give algebra isomorphisms. To see that it preserves the comultiplication, we can check straightforwardly using Equation (\ref{eq5}). We will skip the detailed calculation here.
\end{example} 

Now we define $\mathrm{Hom}^\ast_{S}((\mathbb{Z}/p\mathbb{Z})^2,G^2)$ to be the subfunctor of $\mathrm{Hom}_S((\mathbb{Z}/p\mathbb{Z})^2,G^2)$ as follows:

\begin{definition}\label{full}
Define $\mathrm{Hom}^\ast_S((\mathbb{Z}/p\mathbb{Z})^2,G^2)$ to be the functor whose $T$-valued points are the elements in $\mathrm{Hom}_{S}((\mathbb{Z}/p\mathbb{Z})^2,G^2)(T)=\mathrm{Mat}_2(G(T))$ so that all nonzero $\F_p$-linear combinations of rows and columns are in $(G^2)^\times(T)$. For nonzero $\F_p$-linear combinations, we mean elements like $[m]f\dot{+}[n]g$ where $m$ and $n$ are not both zero in $\F_p$.
\end{definition}

\begin{remark}
It is easy to see that the functor $\mathrm{Hom}^\ast_S((\mathbb{Z}/p\mathbb{Z})^2,G^2)$ we defined above is representable. Indeed, each linear combination being primitive is a closed condition and thus gives a subscheme of $\Hom^\ast_S((\mathbb{Z}/p\mathbb{Z})^2,G^2)=G^4$. Therefore the functor $\mathrm{Hom}^\ast_S((\mathbb{Z}/p\mathbb{Z})^2,G^2)$ is represented by the scheme-theoretical intersection of those subschemes. We use $\Hom^\ast_S((\mathbb{Z}/p\mathbb{Z})^2,G^2)$ for the representing scheme.
\end{remark}

Here are some elementary properties of $\Hom^\ast_{S}((\mathbb{Z}/p\mathbb{Z})^2,G^2)$:
\begin{proposition}\label{prop1}
	Let $S$ be a $\Z_p$-scheme and let $G, G^\prime$ be finite flat commutative group schemes of rank $p$ over S. Let $\mathrm{GL}_2(\F_p)$ act on $\Hom_{S}((\mathbb{Z}/p\mathbb{Z})^2,G^2)$ by acting on $(\mathbb{Z}/p\mathbb{Z})^2$ by precomposition. Then $\Hom^\ast_{S}((\mathbb{Z}/p\mathbb{Z})^2,G^2)$ satisfies:
	\begin{enumerate}[(i)]
		\item By identifying $\Hom_{S}((\mathbb{Z}/p\mathbb{Z})^2,G^2)\times_S T=\Hom_{T}((\mathbb{Z}/p\mathbb{Z})^2,G_T^2)$ for any $S$-scheme $T$, we have $$\Hom^\ast_{S}((\mathbb{Z}/p\mathbb{Z})^2,G^2)\times_S T=\Hom^\ast_{T}((\mathbb{Z}/p\mathbb{Z})^2,G_T^2)$$ as closed subschemes.
		\item The full level structure $\Hom^\ast_{S}((\mathbb{Z}/p\mathbb{Z})^2,G^2)$ is $\mathrm{GL}_2(\F_p)$-invariant. Away from characteristic $p$, we have $$\Hom^\ast_{S[\frac{1}{p}]}((\mathbb{Z}/p\mathbb{Z})^2,G[\frac{1}{p}]^2)=\Isom_{S[\frac{1}{p}]}((\mathbb{Z}/p\mathbb{Z})^2, G[\frac{1}{p}]^2)$$ as closed subschemes of $\Hom_{S[\frac{1}{p}]}((\mathbb{Z}/p\mathbb{Z})^2,G[\frac{1}{p}]^2)$.
		\item Let $\phi: G\to G^\prime$ be an isomorphism and let $\Phi:G^2\to (G^\prime)^2$ be the isomorphism given by $\left( \begin{array}{cc} \phi& 0\\ 0& \phi\\ \end{array} \right).$ Then the isomorphism $\Hom_{S}((\mathbb{Z}/p\mathbb{Z})^2,G^2)\to \Hom_{S}((\mathbb{Z}/p\mathbb{Z})^2,(G^\prime)^2)$ induced by $\Phi$ restricts to an isomorphism on the full level structures $\Hom^\ast_{S}((\mathbb{Z}/p\mathbb{Z})^2,G^2)\to \Hom^\ast_{S}((\mathbb{Z}/p\mathbb{Z})^2,(G^\prime)^2)$.
	\end{enumerate}
\end{proposition}
\begin{proof}\hfill
	\begin{enumerate}[(i)]
		\item It follows straightforwardly from Definition \ref{full} and the fact that the notion of primitive elements is compatible with base change \cite[3.5]{KW}.
		\item Let $f\in \Hom^\ast_{S}((\mathbb{Z}/p\mathbb{Z})^2,G^2)$, regarded as a $2\times 2$ matrix in $G(S)$. Let $g\in \mathrm{GL}_2(\F_p)$. Then $g$ acts on $f$ by $f\mapsto g^tf$, where the scalar multiplication is $[\cdot]$ and the addition is $\dot{+}$. By an elementary calculation, one can see that it suffices to show that if $(u,v)\in (G^2)^\times$ then $(u,v)g\in (G^2)^\times$ for all $g\in \mathrm{GL}_2(\F_p)$. Note that since $G$ is annihilated by $p$, every $m\in\F_p$ defines an endomorphism of $G$ and therefore every $2\times 2$ matrix over $\F_p$ defines an endomorphism of $G^2$ and invertible matrices induce automorphisms of $G^2$. In fact, $(u,v)g$ is the image of $(u,v)$ under the automorphism induced by $g$. Since group scheme automorphisms preserve the augmentation ideal sheaf, they also preserve the primitive elements by Definition \ref{primitive}. Therefore $(u,v)g\in (G^2)^\times$ and we are done.
		
		For the second half of (ii), note that $G[\frac{1}{p}]$ is \'etale locally isomorphic to $\underline{\Z/p\Z}$ and by definition $\Hom^\ast((\mathbb{Z}/p\mathbb{Z})^2,(\underline{\Z/p\Z})^2)=\Isom ((\mathbb{Z}/p\mathbb{Z})^2, (\underline{\Z/p\Z})^2)$. Then the statement is an immediate result of (i).
\item As in (ii), since every group scheme isomorphism preserves the augmentation ideal sheaf, by Definition \ref{primitive} it also preserves the primitive elements. Then it is straightforward to check that (iii) holds by Definition \ref{full}.
	\end{enumerate}	
\end{proof}

Now here is the main result:

\begin{theorem}\label{thm1}
	Let $S$ be a $\Z_p$-scheme and let $G$ be a finite flat commutative group scheme of rank $p$ over S. Let $\Hom^\ast_{S}((\mathbb{Z}/p\mathbb{Z})^2,G^2)$ be as defined in Definition \ref{full}. Then $\Hom^\ast_{S}((\mathbb{Z}/p\mathbb{Z})^2,G^2)$ is flat over $S$ of rank $|\mathrm{GL}_2(\F_p)|$.
\end{theorem}

\section{Proof of the main theorem}
By Proposition \ref{prop1} (i), since being flat is a local property, we can reduce to the case where $S=\spec A$ with $A$ being a local $\Z_p$-algebra \cite[Section 2.4]{Gr}. Recall from Section 2 that the group scheme $G/S$ is determined by a triple $(\mathcal{L},u,v)$. Since $A$ is local, the line bundle $\mathcal{L}$ on $S$ is trivial. Let $\mathcal{A}=\Z_p[s,t]/(st-w_p)$ and $\mathcal{S}=\spec \mathcal{A}$. Let $\mathcal{G}=\spec \mathcal{A}[x]/(x^p-tx)$ be the group scheme over $\mathcal{S}$ with comultiplication 
\begin{equation}\label{eq1}
	m^\ast(x)=1\otimes x+x\otimes 1+\dfrac{1}{1-p}\sum\limits_{i=1}^{p-1}\dfrac{sx^i\otimes x^{p-i}}{w_iw_{p-i}}.
\end{equation}
Then $G/S$ will be the pull back of $\mathcal{G}/\mathcal{S}$ through a morphism $S\to \mathcal{S}$ determined by $u$ and $v$. Applying Proposition \ref{prop1} (i) again, we can see that it suffices to show the flatness of the full level structure for $\mathcal{G}^2/\mathcal{S}$. % By abuse of notation, we will refer $\mathcal{G}/\mathcal{S}$ as the "universal group scheme".

We first look at $\Hom^\ast_{\mathcal{S}}((\mathbb{Z}/p\mathbb{Z})^2,\mathcal{G}^2)$ over the two open subschemes $\spec\mathbb{Z}_p[s,s^{-1}]$ and $\spec\mathbb{Z}_p[t,t^{-1}]$ of $\mathcal{S}$. It is easy to check that after applying \'etale base changes by adding the $p-1$th root of $s,s^{-1},t,t^{-1}$, we get $\mathcal{G}\times_{\mathcal{S}}\spec\mathbb{Z}_p[s^{\frac{1}{p-1}},s^{-\frac{1}{p-1}}]\cong\mu_p$ and $\mathcal{G}\times_{\mathcal{S}}\spec\mathbb{Z}_p[t^{\frac{1}{p-1}},t^{-\frac{1}{p-1}}]\cong\underline{\Z/p\Z}$. In these cases, the following lemma is as expected:

\begin{lemma}{\label{lm2}}
	We have the following two isomorphisms of group schemes:
	\begin{enumerate}[(i)]
		\item $\Hom^\ast_{\mathcal{S}}((\mathbb{Z}/p\mathbb{Z})^2,\mathcal{G}^2)\times_{\mathcal{S}}\spec\mathbb{Z}_p[s^{\frac{1}{p-1}},s^{-\frac{1}{p-1}}]$
		
		\hspace{10em}$\cong\Hom^\mathrm{full}_{\spec \mathbb{Z}_p[s^{\frac{1}{p-1}},s^{-\frac{1}{p-1}}]}((\mathbb{Z}/p\mathbb{Z})^2,\mu_p^2).$
		
		\noindent Here the $\Hom^\mathrm{full}$ is the full level structure for $\mu_p\times \mu_p$ defined by Wake in \cite{Wa}.
		
		\item $\Hom^\ast_{\mathcal{S}}((\mathbb{Z}/p\mathbb{Z})^2,\mathcal{G}^2)\times_{\mathcal{S}}\spec\mathbb{Z}_p[t^{\frac{1}{p-1}},t^{-\frac{1}{p-1}}] \cong \underline{\mathrm{GL}_2(\Z/p\Z)}.$
	\end{enumerate}
\end{lemma}

\begin{proof}
	Note that from the definition of $\Hom^\ast$, we have $\Hom^\ast((\mathbb{Z}/p\mathbb{Z})^2,\mu_p^2)=\Hom^{\mathrm{full}}((\mathbb{Z}/p\mathbb{Z})^2,\mu_p^2)$ as they are defined in the same way. For the \'etale part, note that sections of constant group schemes being primitive exactly means being nonzero. So $\Hom^\ast((\mathbb{Z}/p\mathbb{Z})^2,(\underline{\Z/p\Z})^2)$ consists of the matrices satisfying that nonzero linear combinations of rows and columns are nonzero, thus invertible matrices. Hence $\Hom^\ast((\mathbb{Z}/p\mathbb{Z})^2,(\underline{\Z/p\Z})^2)=\underline{\mathrm{GL}_2(\Z/p\Z)}$ and the claim is immediate from Proposition \ref{prop1} (i).
\end{proof}

To make the full level structure explicit for $\mathcal{G}^2/\mathcal{S}$, it is helpful to use the universal homomorphism for description. Consider the universal base $\mathcal{S}^{\text{univ}}=\spec \mathcal{A}^{\text{univ}}$ where $\mathcal{A}^{\text{univ}}=\mathcal{A}[a,b,c,d]/(a^p-ta, b^p-tb, c^p-tc, d^p-td)$. Then we have $\mathcal{S}^{\text{univ}}=\Hom_{\mathcal{S}}((\mathbb{Z}/p\mathbb{Z})^2,\mathcal{G}^2)$. Let $h\in \mathrm{Hom}_{\mathcal{S}^\text{univ}}((\mathbb{Z}/p\mathbb{Z})^2,\mathcal{G}_{\mathcal{S}^\text{univ}}^2)$ be the universal homomorphism defined over $\mathcal{S}^\text{univ}$, given by $(1,0)\mapsto (a,b)$, $(0,1)\mapsto (c,d)$. % We can write $h=\left( \begin{array}{cc} a&b\\c&d\\ \end{array}\right)$. 
Then $\Hom_{\mathcal{S}}^\ast((\mathbb{Z}/p\mathbb{Z})^2,\mathcal{G}^2)$, as a subscheme of the universal base $\mathcal{S}^\text{univ}$, is cut out by the condition $h\in \mathrm{Hom}^\ast_{\mathcal{S}^\text{univ}}((\mathbb{Z}/p\mathbb{Z})^2,\mathcal{G}_{\mathcal{S}^\text{univ}}^2)$. Therefore, by definition, $\Hom_{\mathcal{S}}^\ast((\mathbb{Z}/p\mathbb{Z})^2,\mathcal{G}^2)$ is given by the ideal $I \subset \mathcal{A}^{\text{univ}}$ generated by 
\begin{align*}
	&\left\{\left(\left([m]a\dot{+}[n]b\right)^{p-1}-t\right)\left(\left([m]c\dot{+}[n]d\right)^{p-1}-t\right),\right.\\
	&\left.\left(\left([m]a\dot{+}[n]c\right)^{p-1}-t\right)\left(\left([m]b\dot{+}[n]d\right)^{p-1}-t\right)\right\}_{(m,n)\in \F_p^2\setminus\{(0,0)\}}.
\end{align*}

Recall that in the notion $[m]a\dot{+}[n]b$, we are regarding $a,b,c,d$ as elements in $\mathcal{G}(\mathcal{S}^\text{univ})$, corresponding to the homomorphisms $$A[x]/(x^p-tx)\to \mathcal{A}[a,b,c,d]/(a^p-ta, b^p-tb, c^p-tc, d^p-td)$$ sending $x$ to $a,b,c,d$. As an abstract group, $\mathcal{G}(\mathcal{S}^\text{univ})$ is given by $$\{x\in \mathcal{A}[a,b,c,d]/(a^p-ta, b^p-tb, c^p-tc, d^p-td)\big| x^p=tx\}$$ with the group structure given by $x\dot{+}y=x+y+\dfrac{1}{1-p}\sum\limits_{i=1}^{p-1}\dfrac{sx^iy^{p-i}}{w_iw_{p-i}}.$ Therefore 
$$[2]a=2a+\frac{1}{1-p}\sum\limits_{i=1}^{p-1}\dfrac{sa^p}{w_iw_{p-i}}=2a+\dfrac{1}{1-p}\sum\limits_{i=1}^{p-1}\frac{sta}{w_iw_{p-i}}
=\left(2+\frac{1}{1-p}\sum\limits_{i=1}^{p-1}\frac{w_p}{w_iw_{p-i}}\right)a.$$
Using Equation (\ref{eq5}), we get $[2]a=\chi(2)a$ and in general by induction we have $[m]a=\chi(m)a$. Therefore the full level structure on $\mathcal{G}^2/\mathcal{S}$ has the following expression:
\begin{align}\label{eq2}
\begin{split}
	&\Hom^\ast_{\mathcal{S}}((\mathbb{Z}/p\mathbb{Z})^2,\mathcal{G}^2)\\
	& \cong  \spec \Z_p[s,t,a,b,c,d] \Bigg/ \left(  \substack{st-w_p, a^p-ta, b^p-tb, c^p-tc, d^p-td, \\  \left\{\left(\left(\chi(m)a\dot{+}\chi(n)b\right)^{p-1}-t\right)\left(\left(\chi(m)c\dot{+}\chi(n)d\right)^{p-1}-t\right),\right.\\ \left.\left(\left(\chi(m)a\dot{+}\chi(n)c\right)^{p-1}-t\right)\left(\left(\chi(m)b\dot{+}\chi(n)d\right)^{p-1}-t\right)\right\}} \right).
\end{split}
\end{align}

Having all these set up, we will prove the flatness of $\Hom^\ast_{\mathcal{S}}((\mathbb{Z}/p\mathbb{Z})^2,\mathcal{G}^2)$ over $\mathcal{S}$ using the lemma below:

\begin{lemma}[\cite{AV} Page 51 Lemma 1]\label{prop2}
	Let $Y$ be a reduced scheme and $\mathcal{F}$ a coherent sheaf on $Y$ such that $\dim_{k(y)}\mathcal{F}\otimes_{\mathcal{O}_y}k(y)=r$, for all $y\in Y$. Then $\mathcal{F}$ is a locally free of rank $r$ on $Y$.
\end{lemma}

Apply Lemma \ref{prop2} to $Y=\mathcal{S}$. Note that for $y\in \spec \Z_p[t, t^{-1}]$, we know that $\dim_{k(y)}\left(\mathcal{O}_{\Hom^\ast_\mathcal{S}}\otimes_{\mathcal{O}_y}k(y)\right)=|\mathrm{GL}_2(\F_p)|$ from Lemma \ref{lm2} (ii) and \'etale descent. For $y\in \spec \Z_p[s, s^{-1}]$, we can get $\dim_{k(y)}\left(\mathcal{O}_{\Hom^\ast_\mathcal{S}}\otimes_{\mathcal{O}_y}k(y)\right)=|\mathrm{GL}_2(\F_p)|$ by combining Lemma \ref{lm2} (i) together with Wake's result on $\Hom^{\mathrm{full}}$ and \'etale descent. The only remaining point is $y_0$ for $s=t=p=0$.

Consider $\mathcal{G}/\mathcal{S}$ modulo $p$, denoted by $\mathcal{G}_p/\mathcal{S}_p$. The underlying base scheme $\mathcal{S}_p=\spec \F_p[s,t]/(st)$ is a union of two affine lines and the concerning point $y_0$ is the origin of $\mathcal{S}_p$. Note that $\chi(m)\equiv m \text{ mod } p$. Therefore, by setting $p=0$ from Equation (\ref{eq2}), we get
\begin{equation}\label{eq3}
	\begin{split}
		&\Hom^\ast_{\mathcal{S}_p}((\mathbb{Z}/p\mathbb{Z})^2,\mathcal{G}_p^2)\\
		&\cong\spec \mathcal{O}_{\mathcal{S}_p}[a,b,c,d] \Big/ \left(  \substack{a^p-ta, b^p-tb, c^p-tc, d^p-td, \\  \left\{\left(\left(ma\dot{+}nb\right)^{p-1}-t\right)\left(\left(mc\dot{+}nd\right)^{p-1}-t\right),\right.\\ \left.\left(\left(ma\dot{+}nc\right)^{p-1}-t\right)\left(\left(mb\dot{+}nd\right)^{p-1}-t\right)\right\}} \right).
	\end{split}
\end{equation} Here the ``$\dot{+}$" operation is given as $x\dot{+}y=x+y+\sum\limits_{i=1}^{p-1}\dfrac{sx^iy^{p-i}}{i!(p-i)!}$ (recall that $w_i\equiv i! \text{ mod p}$ from Section 2). Now we have a key observation on Equation (\ref{eq3}).
\begin{theorem}\label{thm2}
	Let $\mathcal{G}_p/\mathcal{S}_p$ be the ``universal" Oort--Tate group scheme in characteristic $p$ as above. Then the ideal defining the full level structure $\Hom^\ast_{\mathcal{S}_p}((\mathbb{Z}/p\mathbb{Z})^2,\mathcal{G}_p^2)$ as a closed subscheme of $\mathcal{G}_p^4$ is generated by elements which do not involve the parameter $s$.
\end{theorem}
\begin{proof}
	We claim that in the coordinate ring $(\ref{eq3})$, we have $$(ma\dot{+}nb)^{p-1}-t=u\left((ma+nb)^{p-1}-t\right)$$ for some unit $u$. Then it follows that 
	\begin{equation}\label{eq4}
		\begin{split}
			&\Hom^\ast_{\mathcal{S}_p}((\mathbb{Z}/p\mathbb{Z})^2,\mathcal{G}_p^2)\\
			&\cong\spec \mathcal{O}_{\mathcal{S}_p}[a,b,c,d] \Big/ \left(  \substack{a^p-ta, b^p-tb, c^p-tc, d^p-td, \\  \left\{\left(\left(ma+nb\right)^{p-1}-t\right)\left(\left(mc+nd\right)^{p-1}-t\right),\right.\\ \left.\left(\left(ma+nc\right)^{p-1}-t\right)\left(\left(mb+nd\right)^{p-1}-t\right)\right\}} \right).
		\end{split}
	\end{equation} and we are done.

	When $p=2$, since $st=0$ and $a^2=ta$, we simply have $$ma\dot{+}nb=ma+nb+smnab=(ma+nb)(1+sma).$$ Here $1+sma$ is a unit as $(1+sma)^2=1+s^2m^2a^2=1+s^2m^2at=1$.
	
	Now suppose that $p>2$. Let $g(x,y)=\sum\limits_{i=1}^{p-1}\dfrac{x^iy^{p-i}}{i!(p-i)!}$ be a polynomial in $\F_p[x,y]$. Note this polynomial $g(x,y)$ is divisible by $x+y$ as $g(x,-x)=0$ (note that $p$ is odd). Assume $g(x,y)=(x+y)g^\prime(x,y)$. Then $ma\dot{+}nb=(ma+nb)(1+sg^\prime(ma,nb))$. Note that $g^\prime$ has no constant term and $st=0$. So we have $$\left(1+sg^\prime(ma,nb)\right)^p=1+s^pg^\prime(m^pat,n^pbt)=1.$$ Therefore $1+sg^\prime(ma,nb)$ is a unit and we have $$\left(1+sg^\prime(ma,nb)\right)^{p-1}\left((ma+nb)^{p-1}-t\right)=(ma\dot{+}nb)^{p-1}-t$$ as claimed.
\end{proof}

As a consequence of Theorem \ref{thm2}, for any point $y\in \mathcal{S}_p$ away from $y_0$, we have $\dim_{k(y_0)}\mathcal{O}_{\Hom^\ast_\mathcal{S}}\otimes_{\mathcal{O}_{y_0}}k(y_0)=\dim_{k(y)}\mathcal{O}_{\Hom^\ast_\mathcal{S}}\otimes_{\mathcal{O}_y}k(y)=|\mathrm{GL}_2(\F_p)|$. Applying Lemma \ref{prop2}, we finish proving the flatness.

\section{Full level structure on truncated height-one $p$-divisible groups}

	In this section we let $G$ be a truncated $p$-divisible group of height 1 of rank $p^r$ over $S$. Suppose $G[\frac{1}{p}]$ is \'etale locally isomorphic to $\underline{\Z/p^r\Z}$. Therefore $G[p^r]$ is flat of rank $p^r$ over $S$ for $r\le l$ and $G[p]\times_S S[\frac{1}{p}]$ is \'etale locally isomorphic to $\underline{\Z/p\Z}$.

As an easy application of our result, we may give a notion of full level structure on $G^2$ as follows:

\begin{definition}
Let $G/S$ be as above. We define the full level structure on $G^2$ as the fiber product:
\begin{center}
	\begin{tikzcd}
	{\Hom_{S}^\ast((\mathbb{Z}/p\mathbb{Z})^2,G^2])} \arrow[r] \arrow[d]\arrow[dr, phantom, "\square"]  & {\Hom_{S}((\mathbb{Z}/p\mathbb{Z})^2,G^2])=G^4} \arrow[d] \\
	{\Hom_{S}^\ast((\mathbb{Z}/p\mathbb{Z})^2,G[p]^2])} \arrow[r]         & {\Hom_{S}((\mathbb{Z}/p\mathbb{Z})^2,G[p]^2])=G[p]^4}    
	\end{tikzcd}
\end{center}

Here the bottom horizontal map is the closed immersion and the right vertical map is a quadruple product of $p^{l-1}\colon G\to G[p]$.
\end{definition}

\begin{theorem}
Let $G/S$ be as above. The full level structure $\Hom_{S}^\ast((\mathbb{Z}/p^r\mathbb{Z})^2,G^2)$ defined above satisfies the conditions (1)-(3) in Section 1. In particular it is flat over $S$ of rank $|\mathrm{GL}_2(\Z/p^r\Z)|$.
\end{theorem}

\begin{proof}
By the general theory of $p$-divisible groups (for example, \cite[Lemma 1.5 (b)]{Me}), $G\to G[p]$ is (faithfully) flat of rank $p^{l-1}$. Therefore, by definition, the full level structure $\Hom_{S}^\ast((\mathbb{Z}/p^r\mathbb{Z})^2,G^2)$ is flat over $\Hom_{S}^\ast((\mathbb{Z}/p\mathbb{Z})^2,G[p]^2])$ of rank $p^{l-1}$. Theorem \ref{thm1} implies that $\Hom_{S}^\ast((\mathbb{Z}/p\mathbb{Z})^2,G[p]^2])$ is flat of rank $|\mathrm{GL}_2(\Z/p\Z)|$ over $S$. Note that $|\mathrm{GL}_{2}(\Z/p^r\Z)|=p^{l-1}|\mathrm{GL}_2(\Z/p\Z)|$. Therefore, $\Hom_{S}^\ast((\mathbb{Z}/p^r\mathbb{Z})^2,G^2)$ is flat over $S$ of rank $|\mathrm{GL}_2(\Z/p^r\Z)|$. The conditions (2) and (3) are immediate from the definition.
\end{proof}

\section{Nonexistence of full level structure over the stack}
Let $\mathbf{C}$ be a stack of group schemes of certain type over $\mathbf{Sch_{\Z_p}}$. (By a stack here we simply mean a category fibered in groupoids over $\mathbf{Sch_{\Z_p}}$ as in \cite{DM}.) So, we assume that the objects in $\mathbf{C}$ are group schemes $G/S$ of certain fixed type (for example, finite flat commutative and of certain rank) and the morphisms are Cartesian squares. By a full level structure over $\mathbf{C}$, we mean a fibered functor $\mathcal{F}\colon \mathbf{C}\to\mathbf{Sch}$, such that $\mathcal{F}(G/S)$ is a closed subscheme of $\Hom_S((\Z/p^r\Z)^g, G)$ and such that for $f\colon G/S\to G^\prime/ S^\prime$, the morphism $\mathcal{F}(f)\colon \mathcal{F}(G/S)\to \mathcal{F}(G^\prime/ S^\prime)$ is the restriction of the morphism $\Hom_S((\Z/p^r\Z)^g, G)\to \Hom_{S^\prime}((\Z/p^r\Z)^g, G^\prime)$ induced by $f$, and satisfies conditions 
as in (1)-(3) of Section 1. We can see that the condition (4) is automatic by this definition.

In this section, we observe the lack of a good notion of full level structure over the stack of finite flat commutative group schemes. In fact, consider the substack $OT\times OT$, whose objects are $G\times G^\prime$ where $G, G^\prime$ are Oort--Tate schemes. We will see that even on $OT\times OT$, there is no good notion of full level structure:

\begin{theorem}
There is no notion of full level structure over the stack $OT\times OT$.
\end{theorem}

\begin{proof}
Let $\mathcal{G}/\mathcal{S}$ be as in Section 3. Assume there is a full level structure on $OT\times OT$ satisfying (1)-(4). Then the full level structure on $\mathcal{G}^2/\mathcal{S}$ must be the one we defined. In fact over the generic fiber of $\mathcal{S}$, the full level structure is given by the condition (2). Therefore the only way to satisfy condition (1) is defining the full level structure over $S$ as the Zariski closure of the corresponding scheme over the generic fiber. Note that any group scheme of rank $p$ over a local ring can be obtained from $\mathcal{G}/\mathcal{S}$ by base change. Because of condition (3), the full level structure on $G\times G$ over a local base must be the one we defined above. However, this only possible structure is not preserved under all group scheme automorphisms. Here is one example communicated to the author by Wake:

Consider the full level structure on $\alpha_p\times \alpha_p$ over $\overbar{\F}_p$ with $p>2$. By our definition and Theorem \ref{thm2}, we have
$$
\Hom^\ast_{\overbar{\F}_p}((\Z/p\Z)^2, \alpha_p^2)\cong \spec \overbar{\F}_p[a,b,c,d] \Bigg/ \left(  \substack{a^p,b^p,c^p,d^p \\  \left\{\left(ma+nb\right)^{p-1}\left(mc+nd\right)^{p-1},\right.\\ \left.\left(ma+nc\right)^{p-1}\left(mb+nd\right)^{p-1}\right\}} \right).
$$
 Note that $\text{Aut}_{\overbar{\F}_p}(\alpha_p^2)=\text{GL}_2(\overbar{\F}_p)$, with the action given by multiplying 
 $
 \left( \begin{array}{cc} a&b\\c&d\\ \end{array}\right)
 $
  by elements of $\text{GL}_2(\overbar{\F}_p)$ from the right. Since $(m,n)\in \F_p^2\setminus\{(0,0)\}$, it is not hard to see that the ideal is not invariant under the action of $\text{GL}_2(\overbar{\F}_p)$. 
\end{proof}

\begin{remark}
The notion of $\Hom^\ast$ is not preserved under a general isomorphism $G\times G\to G^\prime\times G^\prime$ as observed. However, as we see in Proposition \ref{prop1}, we may restrict the isomorphisms to those of the form $$\Phi=\left( \begin{array}{cc} \phi& 0\\ 0& \phi\\ \end{array} \right)$$ where $\phi\colon G\xrightarrow{\sim} G^\prime$ is an isomorphism.
Roughly speaking, we fail to have a good full level structure over $OT\times OT$, but the one we give behaves well when regarded as a full level structure over $OT$.
\end{remark}

\begin{remark}
Although as shown, a good notion of full level structure on the stack of all finite group schemes does not exist, one might still hope to define a full level structure on truncated $p$-divisible groups. However, some new idea is needed.
\end{remark}

\section{Full level structure on $G^3$}

Another natural question we may ask is whether we can have some similar results for group schemes of the form $G^n$, where $G$ is an Oort--Tate group scheme. We record some partial results here. However, a full answer to this question requires some new idea.

Let us take $G=\mu_p$ over $\spec \Z$. One intermediate step towards defining a full level structure on $G^3$ is defining a ``partial level structure" as a subscheme of $\Hom_{\Z}((\Z/p\Z)^2, (\mu_p)^3)$. We will still require that the resulting scheme is flat over the base and when inverting $p$ we want $\Hom_{\Z}^\ast((\Z/p\Z)^2, (\mu_p)^3)\cong \underline{\mathrm{Mat}^\ast_{2\times 3}(\F_p)}$, where $\mathrm{Mat}^\ast_{2\times 3}$ denote the set of all $2\times 3$ matrices of rank 2. It turns out that this can be done using our result in this paper. Let $h$ be the universal homomorphism. Then $\Hom_{\Z}^\ast((\Z/p\Z)^2, (\mu_p)^3)$ is cut out by the following conditions:

\begin{enumerate}[(i)]
%\item For any $\Z/p\Z \hookrightarrow (\Z/p\Z)^2$ and any $(\mu_p)^3)\twoheadrightarrow \mu_p$, the resulting $\mathrm{Hom}(\Z/p\Z, \mu_p)$ lies in $\mathrm{Hom}^\ast(\Z/p\Z, \mu_p)$.
\item All nonzero linear combinations of rows and columns are primitive.
%item By composing any $(\Z/p\Z)^2 \xrightarrow{\cong} (\Z/p\Z)^2$ and any $(\mu_p)^3\xrightarrow{\cong} (\mu_p)^3$, the resulting $\mathrm{Hom}((\Z/p\Z)^2, (\mu_p)^3)$ lies in $\mathrm{Hom}^\ast(\Z/p\Z, \mu_p)$.
\item After applying any left $\mathrm{GL}_2(\F_p)$-action and right $\mathrm{GL}_3(\F_p)$-action to $h$, one of the three $2\times 2$ blocks of the resulting homomorphism lies in the full level structure $\Hom_{\Z}^\ast((\Z/p\Z)^2, (\mu_p)^2)$.
\end{enumerate}

Let us make (ii) clear here. Let 
\[
h=\left(\begin{array}{ccc} a_{11} & a_{12} & a_{13}\\ a_{21} & a_{22} & a_{23} \end{array}\right)
\]
 be the universal homomorphism. Let $I_1$, resp.~$I_2$, $I_3$, be the ideal defined by requiring that 
 \[
 \left(\begin{array}{cc} a_{11} & a_{12} \\ a_{21} & a_{22} \end{array}\right),\quad
 \hbox{\rm resp.}\left(\begin{array}{cc} a_{11} & a_{13}\\ a_{21} & a_{23} \end{array}\right),  \ \left(\begin{array}{cc} a_{12} & a_{13}\\ a_{22} & a_{23} \end{array}\right),
  \]
   lies in the full level structure subscheme
 $\Hom^\ast((\Z/p\Z)^2, (\mu_p)^2)$. Then the ideal defining ``one of the three $2\times 2$ blocks lies the full level structure" is the ideal $I_1I_2\cap I_1I_3\cap I_2I_3$. The closed subscheme $\Hom_{\Z}^\ast((\Z/p\Z)^2, (\mu_p)^3)$ cut out by these conditions is flat of rank $|\mathrm{Mat}^\ast_{2\times 3}(\F_p)|$ over the base. This result of ``partial level structure" $\Hom_{\Z}^\ast((\Z/p\Z)^2, (\mu_p)^3)$ can be extended to $\Hom^\ast((\Z/p\Z)^2, G^3)$.

One might hope to define $\Hom_{\Z}^\ast((\Z/p\Z)^3, (\mu_p)^3)$ using the ``partial level structure" above, by requiring that after applying the left and right $\mathrm{GL}_3(\F_p)$-action and possibly Cartier dual to the universal homomorphism, the resulting homomorphism is such that any $2\times 3$ block is giving a ``partial level structure". It turns out that this condition is very close to what we want, but still not enough. Here are some numerical results. Consider $\mu_p$ over $\F_p$. For $p=2$, the above condition will give a closed subscheme of rank 169 over $\F_p$, while $|\mathrm{GL}_3(\F_2)|=168$. For $p=3$, the obtained subscheme has rank 11473 over $\F_p$, while $|\mathrm{GL}_3(\F_3)|=11232$ (comparing with $3^9=19683$). So, some further conditions need to be discovered.



\bibliographystyle{amsplain}

\begin{thebibliography}{1}

\bibitem {CN}
C. Chai and P. Norman,
Bad reduction of the {S}iegel moduli scheme of genus two with {$\Gamma_0(p)$}-level structure,
American Journal of Mathematics, Volume 112, 1990, No.6, 1003--1071.

\bibitem {DM}
P. Deligne and D. Mumford,
The irreducibility of the space of curves of given genus,
Inst. Hautes Études Sci. Publ. Math. No. 36, 1969, 75--109.

\bibitem {Dr}
V. Drinfeld,
Elliptic modules,
Mat. Sb. (N.S.), Volume 94(136), 1974, 594--627, 656.

\bibitem{Gr}
A. Grothendieck,
\'{E}l\'{e}ments de g\'{e}om\'{e}trie alg\'{e}brique. {IV},
Inst. Hautes \'{E}tudes Sci. Publ. Math. 1967.

\bibitem {HR}
T. Haines and M. Rapoport,
Shimura varieties with {$\Gamma_1(p)$}-level via {H}ecke algebra isomorphisms: the {D}rinfeld case,
Ann. Sci. \'{E}c. Norm. Sup\'{e}r. (4), Volume 45, 2012, No.5,719--785 (2013).

\bibitem {KM}
N. Katz and B. Mazur,
Arithmetic moduli of elliptic curves,
Annals of Mathematics Studies, Volume 108, Princeton University Press, Princeton, NJ, 1985, xiv+514.

\bibitem {KW}
R. Kottwitz and P. Wake,
Primitive elements for {$p$}-divisible groups,
Research in Number Theory, Volume 3, 2017, Art. 20.

\bibitem {Ma}
H. Matsumura,
Commutative ring theory,
Cambridge Studies in Advanced Mathematics, Volume 8, Translated from the Japanese by M. Reid, Cambridge University Press, Cambridge, 1986, xiv+320.

\bibitem {Me}
W. Messing,
The crystals associated to {B}arsotti-{T}ate groups: with applications to abelian schemes,
Lecture Notes in Mathematics, Vol. 264, Springer-Verlag, Berlin-New York, 1972, iii+190.

\bibitem {AV}
D. Mumford,
Abelian varieties,
Tata Institute of Fundamental Research Studies in Mathematics, No. 5, 1970, viii+242.

\bibitem {OT}
J. Tate and F. Oort,
Group schemes of prime order,
Ann. Sci. \'{E}cole Norm. Sup. (4), Volume 3, 1970, 1--21.

\bibitem {Wa}
P. Wake,
Full level structures revisited: pairs of roots of unity,
Journal of Number Theory, Volume 168, 2016, 81--100.


\end{thebibliography}

\end{document}